\documentclass[12 pt]{amsart}
\usepackage{amscd,amssymb,amsthm}
\usepackage[arrow,matrix]{xy}
\usepackage{graphicx}
\oddsidemargin=0.3in \evensidemargin=0.3in \theoremstyle{plain}
\newtheorem{thm}[subsection]{Theorem}

\newtheorem{exa}[subsection]{Example}

\theoremstyle{definition}
\newtheorem{rk}[subsection]{\textmd{Remark}}

\numberwithin{equation}{section} 

\begin{document}
\title [Solution of a Diophantine Equation ]{Solution of the Diophantine Equation  $  x_{1}x_{2}x_{3}\cdots x_{m-1}=z^n $}
\author{Zahid Raza,  Hafsa Masood Malik}
 \address{Department of Mathematics, \\ National University of Computer and Emerging Sciences, B-Block, Faisal Town, Lahore,
         Pakistan.}
\email{zahid.raza@nu.edu.pk, hafsa.masood.malik@gmail.com}



\begin{abstract}
This work determine the entire family of positive integer solutions of the diophantine equation. The solution is described in terms of $\frac{(m-1)(m+n-2)}{2} $ or $\frac{(m-1)(m+n-1)}{2}$ positive parameters depending on  $n$ even or odd. We find the solution of a diophantine system of equations by using the solution of the diophantine equation. We generalized  all the results of the paper \cite{Kz}.
\end{abstract}
\maketitle
\section{Introduction and Preliminaries}
All the solutions of a diophantine equation of the form
\begin{equation*}
    ax-by=c,
\end{equation*}
has been found. But the theory on this equation in the literature can not apply on a diophantine equation of the form
\begin{equation}\label{e1}
      x_{1}x_{2}x_{3}\cdots x_{m-1}=z^n.
\end{equation}
So, we archived the solution of the equation \ref{e1}.
In \cite{Kz}, the author worked on the diophantine equation\ref{e1} for $ m=3$ with $n=2,3,4,5,6,$ and $m=4$ with $n=2;$ furthermore, he also worked on a diophantine system of $2-$equation of $5-$variables. We extent all these  results to the general case that is for all $m\geq 3$ and $n\geq 2$ and use it to find the solution of a system of $s$-diophantine equation in $t$ variables.

 The authors have not been able to find material on the equations of this paper. However, W. Sierpinski's book, Elementary Theory of Numbers listed two papers with material on equations of the form, $x_{1}x_{2}x_{3}\cdots x_{n}=t^k$. One is a 3-page paper published in 1955 (see reference[3]). And the other, a 10-page paper published in 1933 (see reference[4]). The authors have not been able to access these two papers. Even if some of the results in those two works; overlap with some of the results in this work, it is quite likely, that the methods used in this paper; will be different from those used in the above mentioned papers.
\begin{itemize}
\item An integer $ b $ is called divisible by an other integer $ a \neq 0,$ if there exist some integer $ c $ such that
\begin{equation*}
 b=ac.
\end{equation*}
 Symbolically,
 \begin{equation*}
  a \mid b.
\end{equation*}
\item Let $ a $ and $ b $ be given integer, with at least one of them different from zero. A positive integer $ d $ is called the greatest common divisor of the integers $ a ,b $ ??? \\ $\diamondsuit$ If $d\mid a$ and $ d\mid b.$\\$\diamondsuit$ Whenever there is $c$ such that $c\mid a$ and $c\mid b,$ then $c\leq d.$ \\Symbolically, $ d=\gcd(a,b)$
\item Two integer $ a,b $ are said to be relatively prime if $ \gcd(a,b)=1. $
\item Whenever $ \gcd(a,b)=d; $ then $ \gcd(\frac{a}{d},\frac{b}{d})=1. $
\end{itemize}
The first theorem, well-known and widely used in number theory, is known as Euclid's and for a proof
see reference \cite{W} or reference \cite{K}.
\begin{thm}\label{ti}
Assume that $\alpha, \beta, \gamma$   are positive integer such that $ \gcd(\alpha,\beta)=1,$ and $ \alpha\beta \mid \gamma,$  then
$ \alpha \mid \gamma.$
          \end{thm}
The following theorem can be proved by using the fundamental theorem of arithmetic. It can also be
proved without the use of the fundamental theorem (for example \cite{W}).
\begin{thm}\label{tj}
 Consider that $ \alpha, \beta $  and $ n $ to be positive integers. If $ \alpha^n\mid \beta^n ,$ then
$ \alpha \mid \beta $
\end{thm}
In reference \cite{W}, the reader can find a proof of theorem \ref{t2} that makes use of theorem \ref{ti} and \ref{tj} ; but not the factorization theorem of a positive integer into prime
powers.

    \begin{thm}\label{t2}
     If $ \eta , \alpha, \beta, \gamma $ and $n$ are positive integer such that
      \begin{equation*}
          \alpha\beta = \eta\gamma^n  \ \ \ \ where  \ \ \gcd(\alpha, \beta )=1,
\end{equation*}
 then
 \begin{equation*}
 \alpha = \sigma{\alpha'_{1}}^n,  \ \ \beta=\varsigma {\beta'_{1}}^n, \ \ \ \eta=\varsigma\sigma  \ \ \
and \ \ \ \gamma=\alpha'_{1}\beta'_{1}
\end{equation*}
 where, $\gcd(\alpha'_{1}, \beta'_{1})=1=\gcd(\sigma, \varsigma)$.
 \end{thm}
    \begin{thm}\label{t3}
     Let $k$ be a positive integer, and assume the 3-variables diophantine equation
      \begin{equation*}
          xy = kz^n.
\end{equation*}
Then all positive integer solutions can be described by 2-parametric formulas;\\
$ x=k_{1} {t_{1}}^n , \ \ y= k_{2}{t_{2}}^n, \ \ z= t_{1} t_{2} $ where, $\gcd(k_{1},k_{2})=1=\gcd(t_{1},t_{2})$.
 \end{thm}
    \section{Main Result}
\begin{thm}\label{t4}
 The diophantine equation in $m$-variables
\begin{equation*}
x_{1}x_{2}x_{3}\cdots x_{m-1}=z^n
\end{equation*}
 is equivalent to the $ m+3 $-variable system of equations
 \begin{equation}\label{e2}
    X_{m-1}X_{m-2}x_{1}x_{2}\cdots x_{m-3}=v{Z_0}^n,
             \end{equation}
 \begin{equation}\label{e3}
    w^{n-2}=vd^2
\end{equation}
where $d,X_{m-1}X_{m-2}, Z_0, w, v $ are positive integer variable such that\\
  $x_{m-1}=\theta X_{m-1},\ \ \
  x_{m-2}=\theta X_{m-2},\ \ \
  z=wZ_0 ,\ \ \
  \theta = w d ,\\
  \theta = \gcd(x_{m-1},x_{m-2}),\ \
  w=\gcd(z, \theta),
 \gcd(X_{m-1},X_{m-2})=1=\gcd(Z_0,d).$
      \end{thm}
\begin{proof}
Consider the equation \ref{e1} and let
\begin{equation*}
    \theta=\gcd(x_{m-1},x_{m-2})
\end{equation*}
 such that
\begin{equation}\label{e4}
\Big\{x_{m-1} = \theta X_{m-1} ; x_{m-2}=\theta X_{m-2}, \gcd(X_{m-1},X_{m-2}) = 1 \Big\},
\end{equation}
then from equation \ref{e1} and \ref{e4} we obtain,
\begin{equation}\label{e5}
 \theta^2X_{m-1}X_{m-2}x_{1}x_{2}\cdots x_{m-3}=z^n.
\end{equation}
Assume that $ w = \gcd(z, \theta),$ then
\begin{equation}\label{e6}
\Big\{z=wZ_0 ; \  \theta = wd, \ \gcd(Z_0,d ) = 1\Big\}
\end{equation}
and from equation \ref{e6} and \ref{e5} we further have,
 \begin{equation*}
  w^2d^2 X_{m-1}X_{m-2}x_{1}x_{2}\cdots x_{m-1}=Z_0^nw^n, d^2 X_{m-1}X_{m-2}x_{1}\cdots x_{m-3}=Z_0^nw^{n-2}.
\end{equation*}
Thus
\begin{equation}\label{e7}
 d^2 X_{m-1}X_{m-2}x_{3}\cdots x_{m-3}=Z_0^nw^{n-2}, \ \ \ n\geq 2
\end{equation}
since $\gcd(d,Z_0)=1,$ it follows that $\gcd(d^2,Z_0^n)=1$ which together with equation  \ref{e5} and theorem \ref{ti}; implies that $d^2$ must be a divisor of $w^{n-2}$ that is
\begin{equation}\label{e8}
    w^{n-2}=vd^2, \ \ \  \mbox{for some } \ \ v\in \mathbb{Z}
\end{equation}
 from equations \ref{e8} and \ref{e7}, we have the result.
\end{proof}
\begin{thm}\label{t5}
All the positive solution of the equation \ref{e3} can be describe by the parameters as: 
 \begin{description}
\item[i] If $n$ is odd, then\\
$w =\prod\limits_{i=0}^{\frac{n-3}{2}}r_{2i+1}g^2,$\ \
$d =\prod\limits_{i=0}^{\frac{n-3}{2}}(r_{2i+1})^{i}  g^{n-2},$ \ \
$v = \prod\limits_{i=0}^{\frac{n-3}{2}}(r_{2i+1})^{n-2i-2}.$

  \item[ii] If n is even, then\\
  $w = \prod\limits_{i=0}^{\frac{n-4}{2}}{r}_{2i+1}h,$\ \
  $d = \prod\limits_{i=0}^{\frac{n-4}{2}}({r}_{2i+1})^{i} h^{\frac{n-2}{2}},$ \ \
   $v=\prod\limits_{i=0}^{\frac{n-4}{2}}({r}_{2i+1})^{n-2i-2}.$
   \end{description}
   \end{thm}
\begin{proof}$(i):$ If $n$ is odd and let $ D_{0} = \gcd(w,d),$ then
 \begin{equation}\label{e9}
   w = D_{0}.r_{1}\ \textmd{and} \ d = D_{0}.r_{0}\ \
\textmd{such that}\ \ \gcd(r_{1} , r_{0} ) = 1.
\end{equation}
So from \ref{e2} and \ref{e3}, we have
        $ D_{0}^{n-2}r_{1}^{n-2} = vD^2r_{0}^2$
        \begin{equation}\label{e10}
         D_{0}^{n-4}{r_{1}}^{n-2} = vr_{0}^2
\end{equation}
Since $\gcd(r_{1} , r_{0}) = 1 $, so using theorem \ref{ti}, we have $ r_{0}^2 | D_{0}^{n-4} $. Consider $ D_{1} =  \gcd( D_{0} ,  r_{0}  ),$ then
\begin{equation}\label{e11}
D_0 = D_{1}r_{3}\ \textmd{and} \ r_{0} = D_{1}r_{2} \ \textmd{such that} \  \gcd(r_{3} , r_{2}) = 1.
\end{equation}
Thus $ {D_{1}}^{n-4}{r_{3}}^{n-4}r_{1}^{n-2} = v{D_{1}}^2{r_{2}}^2 $
 \begin{equation}\label{e6}
 {D_{1}}^{n-6}{r_{3}}^{n-4}r_{1}^{n-2} = v{r_{2}}^2
 \end{equation}
where $\gcd(r_{2}, r_{1}r_{3}) = 1 $ using theorem\ref{ti}, we get $ {r_{2}}^2 \mid {D_{1}}^{n-6}. $ Continued in this way and let $ D_{i}=\gcd(r_{2i-2},D_{i-1}) ,$ then
\begin{equation}\label{e12}
D_{i-1} = D_{i}r_{2i+1} \ \textmd{and}  \ \ r_{2i-2} = D_{i}r_{2i} \ \textmd{such that} \ \gcd(r_{2i+1}, r_{2i}) = 1
\end{equation}
we obtain $ D_{i}^{n-2i}{r_{2i+1}}^{n-2i}{r_{2i-1}}^{n-2i-6}\cdots {r_{3}}^{n-2}r_{1}^{n-2} = v{D_{i}}^2.{r_{2i}}^2 $
\begin{equation}\label{e13}
{D_{i}}^{n-2i-2}{r_{2i+1}}^{n-2i}{r_{2i-1}}^{n-4i-6}\cdots {r_{3}}^{n-2}r_{1}^{n-2} = v{r_{2i}}^2.
\end{equation}
 Since $\gcd(r_{2i}, r_{1}r_{3} \ldots r_{2i+1}) = 1 $
and using theorem \ref{ti}, we have $ {r_{2i}}^2 \mid {D_{i}}^{n-2i-4} $ where $ i=0,1,2,\ldots, \frac{n-5}{2} $.
Finally we have
        ${r_{n-5}}^2 \mid D_{\frac{n-5}{2}}$
   as $n$ is odd then\\
$ D_{\frac{n-5}{2}} = r_{n-2} {r_{n-5}}^2, $
$ v = r_{n-2}r_{n-4}^{3}r_{n-6}^{5}\cdots r_{3}^{n-4}r_{1}^{n-2} $
$ D_{\frac{n-5}{2}} =r_{n-2}{r_{n-5}}^2 $ and  $ r_{n-7}=r_{n-2}{r_{n-5}}^3 $\\
substituting backward
 we get the required result, where $g=r_{n-5}$ and $h=r_{n-6}.$
\end{proof}

\begin{rk}\label{rk}
\begin{itemize}
\item  When n is an odd integer $  \theta = \prod\limits_{i=0}^{\frac{n-3}{2}} (r_{2i+1})^{i+1}g^{n} $\\
\item  When n is an even integer $  \theta =  \prod\limits_{i=0}^{\frac{n-4}{2}}({r}_{2i+1})^{i+1}h^{\frac{n}{2}} $
\end{itemize}
\end{rk}

\begin{thm}\label{t6}
 Let us consider the $m$-variables Diophantine equation
\ref{e1}.
\begin{description}
           \item[i] If $n$ is odd, then all the positive solution of this equation can be describe by the parametric formulas   \\
                    $x_{j} =\prod\limits_{i=0}^{\frac{n-3}{2}}\Big(k_{2i+1}^{j-1}\Big)^{n-2i-2}\Big(\prod_{t=1}^{j-1}\Big(\gamma^{j-t}_t\Big)\prod_{t=1}^{j}\Big(\gamma^{j-t}_j\Big)\Big)^{n-1}\gamma^{j-t}_j\eta_j^{m-2-j}$\\ $j=1,2,\ldots, m-3$\\
                     $x_{m-2} =(\prod\limits_{t=1}^{m-3}(\prod\limits_{i=0}^{\frac{n-3}{2}} ({(k_{2i+1}^{t-1}l^{m-3}_{2i+1})}^{i+1})({k_{2i+1}^{m-3}})^{n-2i-2})({\gamma_t^{m-2-t}})^{n-1}) s_{2}^n  g^{n}$\\
                      $ x_{m-1} = (\prod\limits_{t=1}^{m-3}(\prod\limits_{i=0}^{\frac{n-3}{2}}({(k_{2i+1}^{t-1}l^{m-3}_{2i+1})}^{i+1}){l^{m-3}_{2i+1}}^{n-2i-2})(\eta_t^{m-2-t})^{n-1}) s_{1}^n  g^{n}$\\ $z=\prod\limits_{\lambda=1}^{m-3}\prod\limits_{i=0}^\frac{n-3}{2}(k_{2i+1}^{\lambda-1}l_{2i+1})(\prod\limits_{i=1}^{m-3}\prod\limits_{\lambda=0}^{m-2-i}\eta_i\gamma_i^{\lambda})s_{1}s_{2}g^{2}\\$
\item[ii] If $n$ is even, then all the positive solution of this equation can be describe by the parametric formulas\\
 $x_{j} =\prod\limits_{i=0}^{\frac{n-4}{2}}\Big(k_{2i+1}^{j-1}\Big)^{n-2i-2}\Big(\prod_{t=1}^{j-1}\Big(\gamma^{j-t}_t\Big)\prod_{t=1}^{j}\Big(\gamma^{j-t}_j\Big)\Big)^{n-1}\gamma^{j-t}_j\eta_j^{m-2-j}$\\\\
  for all $j=1,2,\ldots, m-3$\\
                     $x_{m-2} =(\prod\limits_{t=1}^{m-3}(\prod\limits_{i=0}^{\frac{n-3}{2}} ({(k_{2i+1}^{t-1}l^{m-3}_{2i+1})}^{i+1})({k_{2i+1}^{m-3}})^{n-2i-2})({\gamma_t^{m-2-t}})^{n-1}) s_{2}^n  h^{\frac{n}{2}}$\\
                      $ x_{m-1} = (\prod\limits_{t=1}^{m-3}(\prod\limits_{i=0}^{\frac{n-4}{2}}({(k_{2i+1}^{t-1}l^{m-3}_{2i+1})}^{i+1}){l^{m-3}_{2i+1}}^{n-2i-2})(\eta_t^{m-2-t})^{n-1}) s_{1}^n  h^{\frac{n}{2}}$\\ $z=\prod\limits_{\lambda=1}^{m-3}\prod\limits_{i=0}^\frac{n-3}{2}(k_{2i+1}^{\lambda-1}l^{m-3}_{2i+1})(\prod\limits_{i=1}^{m-3}\prod\limits_{\lambda=0}^{m-2-i}\eta_i^{\lambda}\gamma_i^{\lambda})s_{1}s_{2}g^{2}\\$
        \end{description}
  where $  k^t_{2i+1}, l^{m-3}_{2i+1}, b_i^{t}a_i^{t}, s_{1}, s_{2}, h$ and $ g $ are positive integers such that $1= \gcd(s_{1},s_{2})$,
 $ \gcd(k_{2i-1},k^t_{2i-1}l^{m-3}_{2i-1})=1=\gcd(\gamma_i^{t}\eta_i,\gamma_i)=1 \  \forall \ i=1, 2, \ldots, \frac{n-6}{2}, \  t=1, 2, \ldots, m-3.$
 \end{thm}

   \begin{proof}
 Since $ n \geq 2 $, then by theorem \ref{t4},
the given equation is equivalent to the diophantine system, $v=a_0,$
$ X_{m-1}X_{m-2}x_{1}x_{2}x_{3}....x_{m-3}=Z_0^na_0 $ and $
         w^{n-2}=vd^2.$ Thus by theorem \ref{t5} and remark \ref{rk}
let $ P_{1} = \gcd(x_{1},Z_0),$
then $ x_{1} = P_{1}.X_{1} $ and $ Z_0=P_{1}.Z_{1},$ so $
 X_{m-1}X_{m-2}X_{1}x_{2}\cdots x_{m-3} = Z_{1}^n P_{1}^{n-1}a_0. \\ $ Again let $\gcd(Z_{1} , X_{1})=1,$
then, we have
$ X_{1}\mid a_0P_{1}^{n-1}\Rightarrow
  X_{1}.a_{1} =a_0P_{1}^{n-1}$ and by theorem \ref{t2}
$ X_{1} = \alpha_1 \gamma_1^{n-1} $ and
$ a_{1} = \beta_1 \eta_1^{n-1}$
 such that $\alpha_1\beta_1=a_0$ and $ \gamma_1\eta_1=P_1$ where  $\gcd(\alpha_1,\beta_1)=1=\gcd(\gamma_1,\eta_1).$ \\We have $X_{m-1}X_{m-2}x_{2}x_{3}....x_{m-3}=Z_0^na_1,$
   continued in this way and
let $ P_{i} = \gcd(x_{i},Z_{i-1}).$
then $ x_{i} = P_{i}.X_{i} $ and $ Z_{i-1}=P_{i}.Z_{i}.$ Thus
 $X_{m-1}X_{m-2}X_{i}\prod\limits_{j=i+1}^{m-3} x_{j} = Z_{i}^n P_{i}^{n-1}a_{i-1}. $ Since $\gcd(Z_{i} , X_{i})=1$
then,
$ X_{i}\mid a_{i-1}P_{i}^{n-1} \Rightarrow
  X_{i}.a_{i} =a_{i-1}P_{i}^{n-1}$
 so by theorem \ref{t2}, we get
$ X_{i} = \alpha_i \gamma_i^{n-1} $ and $ a_{i} = \beta_i \eta_i^{n-1}$
 such that $\alpha_i\beta_i=a_i$ and $\gamma_i\eta_i=P_i$ where $\gcd(\alpha_i,\beta_i)=1=\gcd(\gamma_i,\eta_i).$  We have $X_{m-1}X_{m-2}\prod\limits_{j=i+1}^{m-3} x_{j}=Z_{i-1}^na_i \ \ \ \forall \ i=1,2,3,\ldots,m-3$ and at last we get
 $X_{m-1}X_{m-2}= Z_{m-3}^n a_{m-3}$
then by theorem \ref{t2}
$ X_{m-2} = \alpha_{m-2} s_1^n $ and
$ X_{m-1} = \beta_{m-2} s_2^n$
 such that $\alpha_{m-2}\beta_{m-2}=a_{m-3} $ and $s_1s_2=Z_{m-4}$ with $\gcd(\alpha_{m-2},\beta_{m-2})=1=\gcd(s_1,s_1).$ Now by using theorem \ref{t2} and technique in proof of theorem \ref{t5}, we have for all $ j=1,2,3,\ldots,m-2$\\ $\alpha_{j}=(\prod\limits_{i=1}^{m-3}(k^{j-1}_{2i+1})^{n-2i-2})(\prod\limits_{t=1}^{j}\gamma^{j-t}_t)^{n-1},\ \ \beta_{j}=(\prod\limits_{i=1}^{m-3}(l^{j-1}_{2i+1})^{n-2i-2})(\prod\limits_{t=1}^{j}\eta^{j-t}_t)^{n-1}.\\$
 Put in equation\ref{e3} and \ref{e5} we get required result.
 \end{proof}
\begin{rk}
The solution is described in terms of $\frac{(m-1)(m+n-2)}{2} $ or $\frac{(m-1)(m+n-1)}{2}$ positive parameters depending on  $n$ even or odd.
\end{rk}
Let  $p_{1}, p_{2}, \ldots, p_{s} $ and $r$ be positive integers and suppose that $ t=p_{1}+p_{2}+\cdots + p_{s} - r$
\begin{thm}\label{t7}
   Consider the $t$-variables diophantine system of $s$ equations,

$ x_{11}x_{12}x_{13}....x_{1p_{1}-1}=z_{1}^{k_{1}}\\
 x_{21}x_{22}x_{23}....x_{2p_{2}-1}=z_{2}^{k_{2}}\\
 \vdots\\
 x_{i1}x_{i2}x_{i3}....x_{ip_{i}-1}=z_{i}^{k_{i}}\\
 \vdots \\
 x_{s1}x_{s2}x_{s3}....x_{sp_{s}-1}=z_{s}^{k_{s}}\ $ where $ 2 \leq k_{j}' $ s are  positive number   where $ s$  no.of equations and $r$ is no of repeated variables in this system.
Then all the positive solutions of this system of equations can be describe by using the following algorithm.
\end{thm}
 \begin{proof}
\begin{itemize}
\item  Write solution of each equation by using theorem \ref{t6}
\item Select one variable from $r$-repeated variables and find the unique $d$ in solution by using technique theorem \ref{t6}
\item Replace those values of parameters appear in selected repeated, in other variable.
\item Do the same activity with other repeated variable.
\item If all repeated variables have unique solution then substituting  the values of parameter that exist in that variables in other variable.
\end{itemize}
        \end{proof}

To explain above theorem, here is two illustration below
\begin{exa}
   Consider the 6-variables
    Diophantine system,
$ x_{1}x_{2}x_{3} = z_{1}^3$ and $
  x_{3}x_{4} = z_{2}^2.$ Then all the positive solutions of this system equations can be describe by $ 11 $ parametric formulas \\
 $ x_{1} =a{k'_1}^2l'_1 {\gamma'_1}^2  R_{1}^3  g^3, \ \
x_{2} = ak'_{1}{l'_{1}}^2{\eta'_1}^2   R_{2}^3  g^3,\ \
x_{3} = a\gamma'_1\eta'_1b_1^3f^6,\ \
 x_{4} = a\gamma'_1 \eta'_1b_1^3r_1^2, \\
z _{1}= al'_{1}k'_{1}\gamma'_1 \eta'_1b_1f^2R_{1}R_{2}g^2,\ \
z _{2}= a\gamma'_1\eta'_1b_1^3f^3r_{1}.$\\
\textbf{\textsc{Solution:-}}\\
Step 1: We apply theorem \ref{t6} to get the following solution for equation $ 1 $ of the system\\
$ \big \{ x_{1} =k_1{k'_1}^2l'_1 {\gamma'_1}^2  R_{1}^3  g^3, \ \
x_{2} = k_1k'_{1}{l'_{1}}^2{\eta'_1}^2   R_{2}^3  g^3,\ \
x_{3} = k_1\gamma_1^3\gamma'_1\eta'_1,\ \
z _{1}= l'_{1}k'_{1}k_{1}\gamma_1\gamma'_1 \eta'_1R_{1}R_{2}g^2\big \} $ \\
and for equation $ 2 $ of the system\\
$ \big \{ x_{4} = dr_1^2,\ \
x_{3} = r_{2}^2 d,\ \
z _{2}= dr_{1}r_{2}\big \} $
Step 2: Since $x_3$ is the repeated variable in the equations of the system, so
\begin{equation}\label{egi}
  r_{2}^2 d=x_{3} = k_1\gamma_1^3\gamma'_1\eta'_1
\end{equation}
Now we will find that unique $d$ for $x_{3}$, and suppose $ a=\gcd(d,k_{1})$
 such that
\begin{equation*}\label{egj}
\Big\{d = aD_1; k_1=aK_1,gcd(D_1,K_1) = 1 \Big\}
\end{equation*}
put in equation \ref{egi}
\begin{equation}\label{egh}
r_{2}^2D_1 =\gamma_1^3\gamma'_1\eta'_1K_1 \ \Rightarrow  D_1\mid \gamma_1^3\gamma'_1\eta'_1
\end{equation}
then there exist a positive integer $b$ such that $D_1b  = \gamma_1^3\gamma'_1\eta'_1$
by using theorem \ref{t6} $D_1= \beta_1c_1^3,\ \ \ b= \beta_2c_2^3,\ \ \
 \gamma_1=c_1c_2,\ \ \gamma'_1\eta'_1=\beta_1\beta_2\ \ \textmd{such that}\ \gcd(c_2,c_1)=1=\gcd(\beta_1,\beta_1)$
 now equation \ref{egh} become
 \begin{equation}\label{egk}
r_{2}^2=\beta_2c_2^3K_1
\end{equation}
from equation \ref{egk} $ r_2^2\mid c_2^3$
then there exist a positive integer $e$ such that $r_2^2e  = c_2^3$ then by using theorem \ref{t5}
\begin{eqnarray*}
c_2 = \alpha_1\alpha_3f^2,  \ \
r_2 = \alpha_3f^3, \ \
e = \alpha_1^3\alpha_3
\end{eqnarray*}
 substituting the values in equation \ref{egk}, we get
$1= \beta_2\alpha_1^3\alpha_3K_1 \Longrightarrow \beta_2=\alpha_1=1=\alpha_3=K_1 $
replacing the values we get the required result.
\end{exa}
\begin{exa}
   Consider the 6-variable diophantine system,
$ x_{1}x_{2}x_{3} = z_{1}^2$ and $
  x_{3}x_{4} = z_{2}^2 .$ Then all the positive solutions of this system of equations can be describe by $ 11$ parametric formulas\\
$x_{1} = s_{1}^2  adr_1 h,\ \
x_{2} = s_{2}^2  bdr_2 h,\ \
x_{3} = ab(c_1r_1r_2dc)^2,\ \
x_{4} = ab(cS_{1})^2,\\
z _{1}= abhs_{1}s_{2}cc_1r_1r_2d^2,\ \
z _{2}= abS_{1}r_1r_2d c_1c^2.$ \\
\textbf{\textsc{Solution:-}}\\
Step 1: We apply theorem \ref{t6} to get the following solution for equation $ 1 $ of the system\\
$ \big \{ x_{1} = s_{1}^2  k_{1} h,\ \
x_{2} = s_{2}^2  l_{1} h,\ \
x_{3} = \gamma^2 k_{1}l_{1},\ \
z _{1}= s_{1}s_{2}\gamma k_{1} l_{1}h\big \} $ \\
and for equation $ 2 $ of the system we get\\
$ \big \{ x_{4} = g S_{1}^2,\ \
x_{3} = S_{2}^2 g,\ \
z _{2}= gS_{2}S_{1} \big \} $\\
Step 2: Since $x_3$ is the repeated variable in the equations of the system, so
\begin{equation}\label{eg1}
 x_{3} = S_{2}^2 g = \gamma^2 k_{1}l_{1}
\end{equation}
Now we will find that unique $d$ for $x_{3}$, and suppose  $ a=\gcd(g,k_{1})$
 such that
\begin{equation}\label{eg2}
\Big\{g = a\alpha; k_1=a\beta,\gcd(\alpha,\beta) = 1 \Big\}
\end{equation}
then equation \ref{eg1} becomes $S_{2}^2\alpha =\gamma^2 \beta l_{1}\Rightarrow \alpha\mid\gamma^2 l_{1}$
so there exist a positive integer $e$ such that $\alpha e  = \gamma^2 l_{1}$.
Now using theorem \ref{t2}, we get $\alpha_1= bc^2,\ \ \ e= b_1c_1^2,\ \ \
 \gamma=cc_1,\ \ \l_1=bb_1\ \textmd{such that}\ \gcd(b,b_1)=1=\gcd(c,c_1).$
 Thus equation \ref{eg1} becomes
 $S_{2}^2=c_1^2\beta b_1$
 clearly $c_1^2\mid S_2^2$ and by theorem \ref{tj}, we have $c_1\mid S_2\Rightarrow _2=c_1f$. Thus we get
 $f^2=\beta b_1$ and by using theorem \ref{t6} we have $f=r_1r_2d ,\ \ \beta=dr_1 ,\ \ b_1=dr_2 $
substituting the values we get the required result.
  \end{exa}


\begin{thebibliography}{00}
\bibitem {W} W.sierpinski, Elementary Theory of Numbers, Warsaw,Poland, $ 1964$
\bibitem {K} Kennerth H. Rosen, Elementary Theory of Numbers and its applications, $ 5^{th} $ edition
\bibitem {A} A.Schinzel, On the equation, $ x_{1}x_{2}\cdots x_{n} = t^{k}, $ Bill. Acad. Polon. Sci. Cl.lll, 3(1955), p.p. 17-19.
\bibitem {M} M.Ward, A type of muplticative diophanitine systems, Amer.J.Math., 55(1933), pp. 67-76.
\bibitem {Kz}Konsatantine Zelator, The diophanitine equation $ xy=z^n $ for $ n=2,3,4,5,6 $; The Diophanitine Equation $ xyz=w^2,$ and the Diophanitine systems $ \big\{ xy=v^2 \ \ yz=w^2\big\}$
\end{thebibliography}
\end{document}